\documentclass{article}
\usepackage[utf8]{inputenc}

\title{
Tetranacci Identities With
Squares, Dominoes, And Hexagonal Double-Strips 
}

\author{Ziqian (Alexa) Jin}
\date{June 2019}

\usepackage{natbib}
\usepackage{graphicx}
\usepackage{float}
\usepackage{amsthm}
\usepackage{amsmath}
\usepackage{caption}

\usepackage{hyperref} 
\hypersetup{
    colorlinks=true,
    linkcolor=magenta,
    filecolor=magenta,      
    urlcolor=magenta,
}

\newtheorem{theorem}{Theorem}

\begin{document}

\maketitle

\section{Abstract}
We combinatorially prove Tetranacci, Tetranacci-Fibonacci, and additional identities using only squares and dominoes on a hexagonal double-strip. Some of these are new proofs of old identities, and others
we believe have never been seen before.

\section{Introduction}

\indent 
As is well known,  the Fibonacci numbers count the number of tilings of a {\em rectangular} strip of length $n$ with squares and dominoes (or in other words, with {\em squares} and {\em double-squares}). There's a book \emph{Proofs That Really Count} by Benjamin and Quinn [1], which has numerous combinatorial proofs of Fibonacci identities.

Similar identities for the Tetranacci numbers can be found by using tilings of a rectangular strip of length $n$ with tiles of lengths 1, 2, 3, and 4 [4]. In this paper, we
look instead at a {\em hexagonal} double-strip with tiles of 
lengths only $1$ (hexagons) and $2$ (double-hexagons). 

First, some definitions. We define an $n$-hexagonal double-strip (later simplified as $n$-HD-strip) as a strip with two rows of a total of $n$ adjacent hexagons. Here is an example of an $8$-HD-strip:

\includegraphics[width=3in]{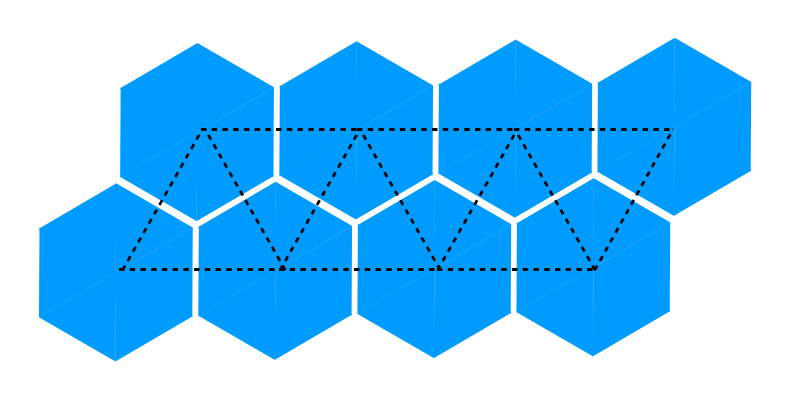}

In order to simplify the graph, we use points to represent the hexagons and dotted lines to indicate the adjacency of two hexagons. 

\includegraphics[width=3in]{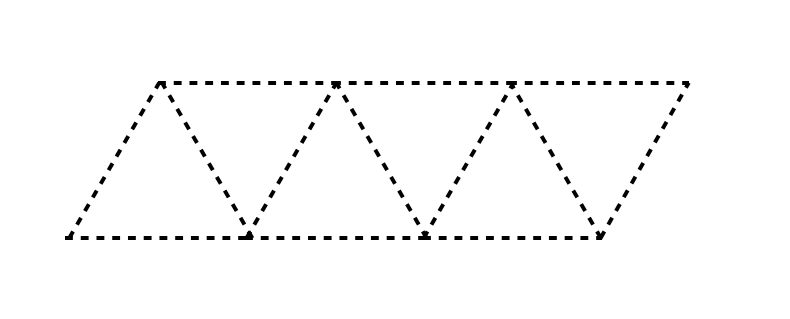}

 In addition, we use squares to represent hexagons and dominoes to represent double-hexagons to make the illustration similar to that in the combinatorial proofs of Fibonacci numbers in [1]. Here is a correspondence between the two graphs.

\includegraphics[width=3in]{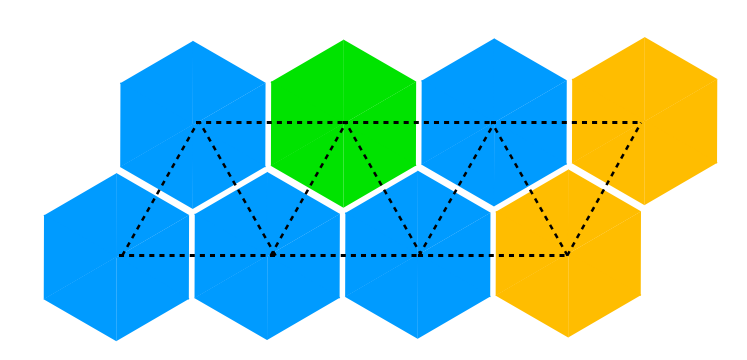}

The hexagon and double-hexagon on an HD-strip, which are pictured above, can be simplified to squares and dominoes on a dot-line frame, which is pictured below.

\includegraphics[width=3in]{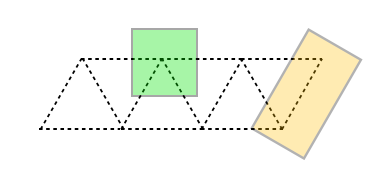}

Now, let us define $T_n$ to be the number of different ways to tile an $n$-HD-board  with
squares and dominoes. Thus, $T_4=8$, because there are eight ways to tile an HD-board of length $4$, as shown here:

\includegraphics[width=4in]{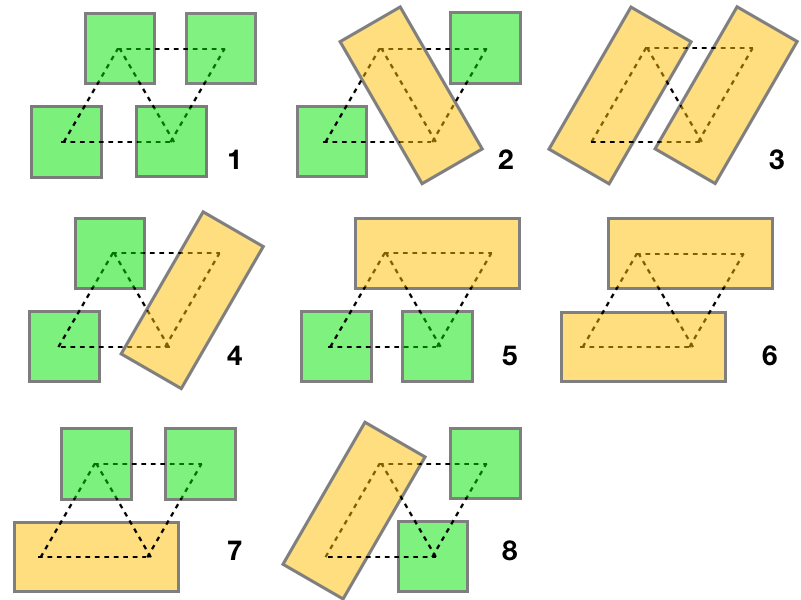}

With a few minutes of effort, we are able to find these values of $T_n$:

\begin{tabular}{c|ccccccccc}
\multicolumn{10}{c}{\rule[-5mm]{0mm}{12mm} \bf The first few values for $T_n$}\\
$n$ & 0 &1&2&3&4&5&6&7 &8  \\
\hline
$T_n$ & 1 & 1 & 2 & 4 & 8 & 15 & 29 & 56 & 108 \\
\end{tabular}
\ 

Because there is exactly one way to tile an HD-strip of length 0, we define that $T_0$ equal $1$.

From the table, we can easily find that for $3<n<9$,
\[ 
T_n = T_{n-1} + T_{n-2} + T_{n-3} +T_{n-4},
\]
and as we will show in a moment, 
this formula holds for all $n > 3$, giving us the 
well-known Tetranacci numbers. 

In the proof of our first theorem, as with most proofs in this paper, one of the two answers to the counting question breaks the problem into disjoint cases depending on a property. We refer to this as conditioning on that property.

Here are some basic definitions that we may refer to later in the proofs of this paper:

\noindent{\bf Cells}: places that will be covered with tiles. The cells in the graphs are indicated by dots.

\noindent{\bf $n$-HD-strip or $n$-HD-board}: hexagonal double-strip with $n$ cells.
 
\noindent{\bf $n$-Single strip}: a rectangular single row of $n$ cells, the visual representation used to prove combinatorial identities of Fibonacci Sequence in [1].
 
\noindent{\bf Cell number for an HD-strip}: count cell number from left to right, as illustrated by the graph. Note that the lower row has all cells with odd cell numbers, and that the upper row has all cells with even cell numbers.
 
\includegraphics[width=3in]{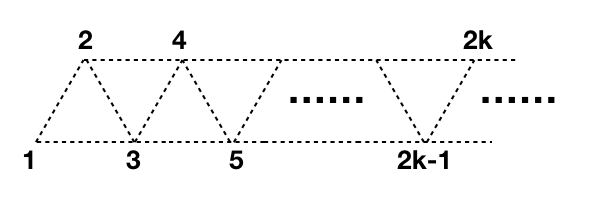}
 
\noindent{\bf Location of a square}: number of the cell it covers.
 
\noindent{\bf Left-domino}: left-inclined domino.
 
\includegraphics[width=1.2in]{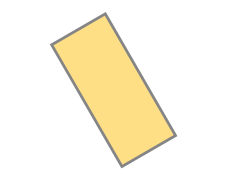}
 
\noindent{\bf Right-domino}: right-inclined domino.
 
\includegraphics[width=1in]{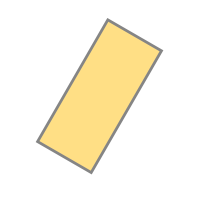}
 
\noindent{\bf Inclined domino}: either left-inclined or right-inclined domino.
  
\noindent{\bf Horizontal  domino}:
 
\includegraphics[width=1.2in]{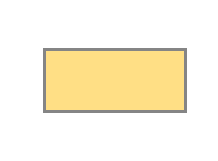}
 
\noindent{\bf Location of a domino}: number of the latter cell it covers. Note that the location of a left-domino is always an odd number bigger than or equal to $3$, the location of a right-domino is always an even number bigger than or equal to $2$, and the location of a horizontal domino is always an integer bigger than or equal to $3$.
 
\noindent{\bf Right-stacked dominoes}:
 
\includegraphics[width=1.5in]{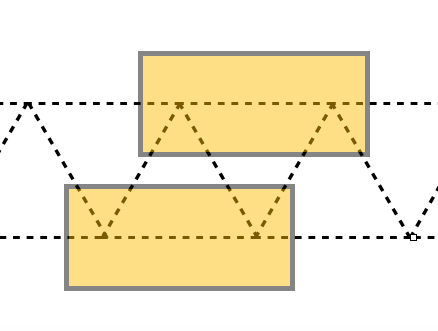}
 
\noindent{\bf Left-stacked dominoes}:
 
\includegraphics[width=1.5in]{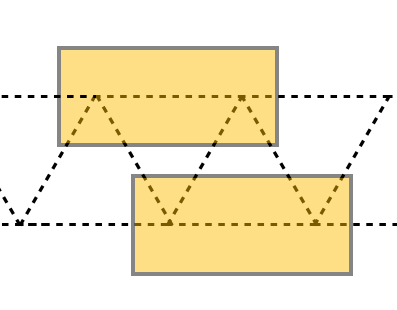}
 
\noindent{\bf Stacked dominoes}: either right-stacked or left-stacked dominoes.
 
\noindent{\bf Nth Diagonal}: the line that separates cells $n$ and $n-1$ from cells $n+1$ and $n+2$.
 
\includegraphics[width=3in]{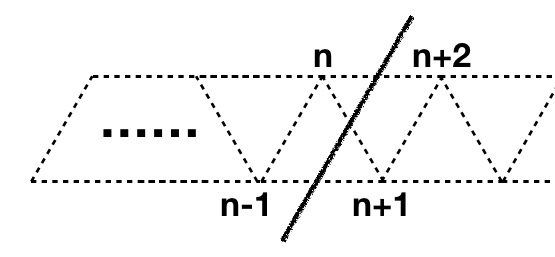}
 
\noindent{\bf Breakability at a given diagonal}: we call a diagonal breakable if no domino lies across it, and unbreakable if at least one domino lies across it. 

\begin{theorem}
For $n>3$,
\[ 
T_n = T_{n-1} + T_{n-2} + T_{n-3} +T_{n-4} 
\]
\end{theorem}

\begin{proof}

Question: How many tilings of an $n$-HD-board exist?
 
Answer 1: By definition, there are $T_n$ such tilings.
 
Answer 2: Condition on the last tile of the $n$-HD-strip. If the last tile is:
 
\begin{enumerate} 
\item a square, there is an ($n$-1)-HD-board left, leading to $T_{n-1}$ tilings.
 
\includegraphics[width=3in]{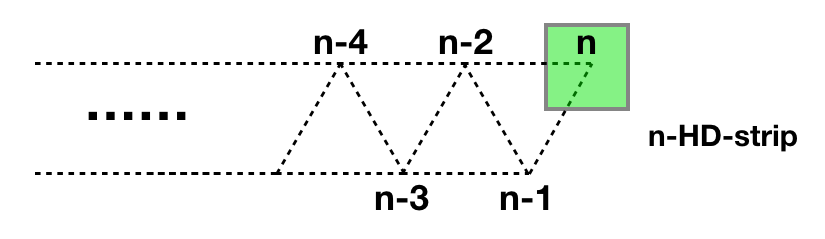}
 
\item an inclined domino: there is an $(n-2)$-HD-board left, leading to $T_{n-2}$ tilings.
 
\includegraphics[width=3in]{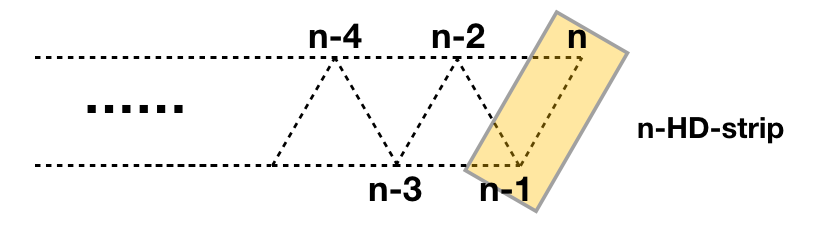}
 
\item a horizontal domino, condition on the tile on cell $(n-1)$. If the tile is:
 
 \begin{enumerate}
\item  a square, there is an $(n-3)$-HD-board left, leading to $T_{n-3}$ tilings.
 
\includegraphics[width=3in]{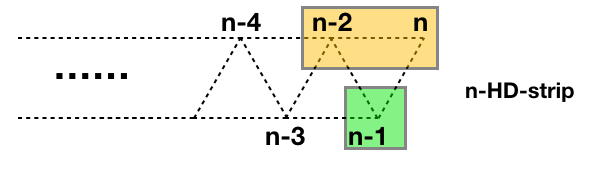}
 \indent
\item  a horizontal domino, there is an $(n-4)$-HD-hoard left, leading to $T_{n-4}$ tilings.

\includegraphics[width=3in]{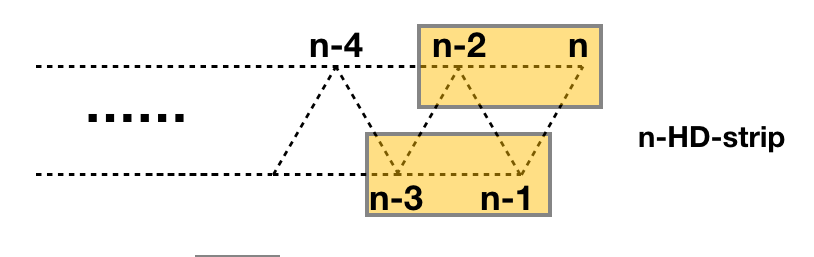}
 \end{enumerate}

\end{enumerate}
\end{proof}

By proving the theorem above, we verify that the sequence  $T_n$ satisfies the property that the sum of four consecutive terms is equal to the next term. Such sequence has an official name: the Tetranaccis [5]. Note that our sequence starts with $T_0=1$, $T_1=1$, $T_2=2$, and $T_3=4$. Particularly, we define $T_{-1}$ to be $0$, because $T_{-1} = T_3 - T_2 - T_1 - T_0$.

\section{New Proofs of Old Theorems}

For these next theorems, we use the combinatorial technique of finding a correspondence between two sets of objects. The formula was proved algebraically in [3] and combinatorially in [4] using tiles of lengths $1,2,\dots,r$ on a rectangular single strip. But here we propose a novel combinatorial approach.

\begin{theorem}
For $n>5$, we have:
\[
2T_{n-1} =T_n + T_{n-5} 
\]
\end{theorem}

\begin{proof}
We first create two sets.
 
Set 1: Tilings of an $(n-1)$-HD-strip. This set has size $T_{n-1}$.
 
Set 2: Tilings of an $n$-HD-strip or an $(n-5)$-HD-strip. This set has size $T_n + T_{n-5}$
 
Correspondence: to prove the theorem, we establish a 1-to-2 correspondence between Set 1 and Set 2.
Specifically, we use each $(n-1)$-HD-strip in Set 1 to create two separate HD-strips in Set 2 that have length $n$ or $(n-5)$, and we do so in such a way that every 
possible $n$-HD-strip and 
$(n-5)$-HD-strip is created exactly once. 

Given a particular $(n-1)$-HD-strip, we first use it to create an  $n$-HD-strip by appending a square to the $(n-1)$-HD-strip.
 
\includegraphics[width=3in]{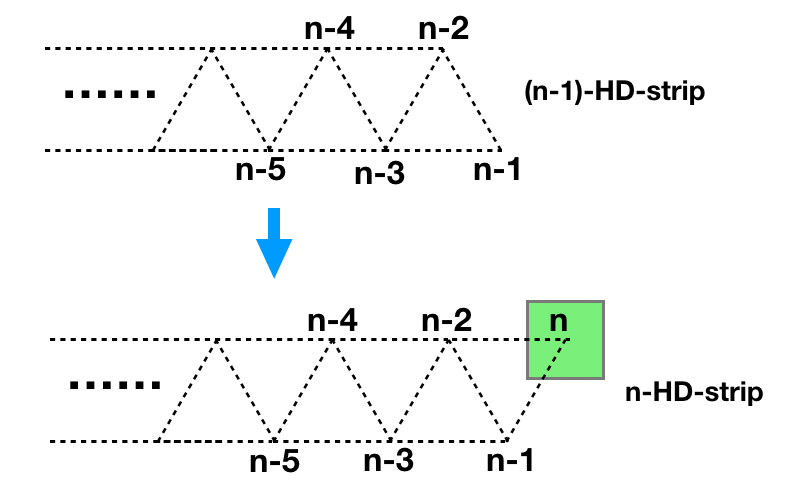}
 
Now, given that same $(n-1)$-HD-strip, 
we use it a second time to create one of the following, depending on the ending of that $(n-1)$-HD-strip. 

If the $(n-1)$-HD-strip ends with:

\begin{enumerate}

\item a square, remove that square and append an inclined domino to the resultant $(n-2)$-HD-strip.
 
\includegraphics[width=3in]{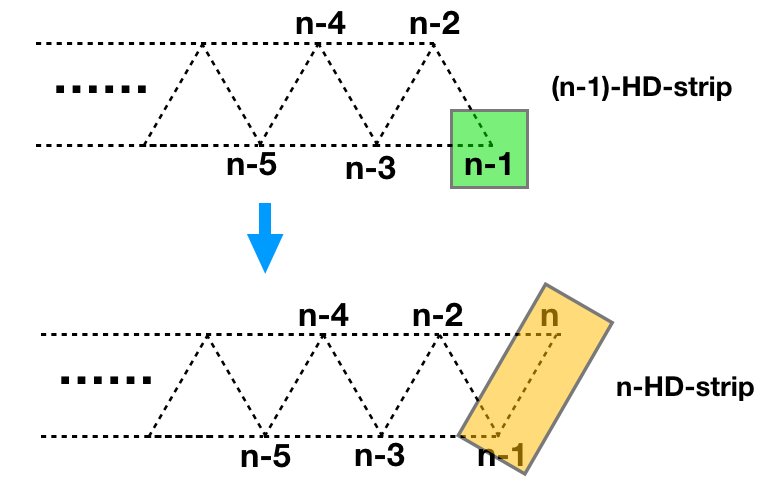}
 
\item an inclined domino, remove that domino and append a square and a horizontal domino to the end of the resultant $(n-3)$-HD-strip, as illustrated by the graph.
 
\includegraphics[width=3in]{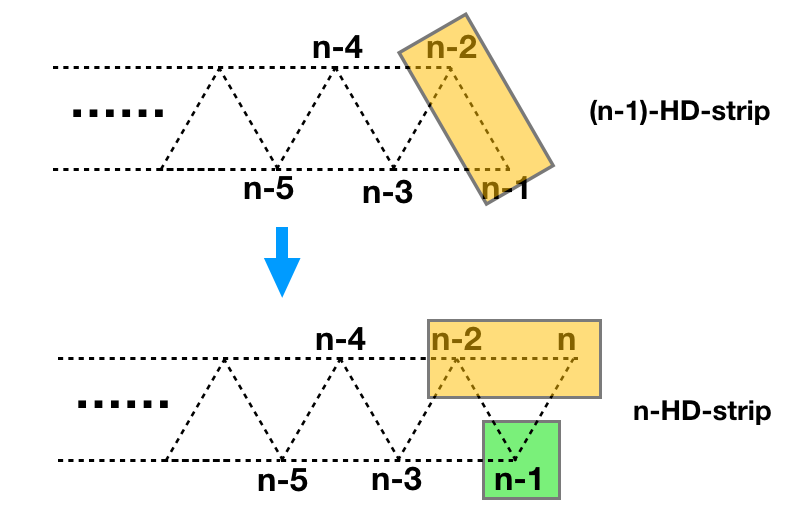}
 
 \item a horizontal domino, if the end: 
 \begin{enumerate}

 \item is two stacked dominoes, remove them both to get an $(n-5)$-HD-strip.
 
 \includegraphics[width=3in]{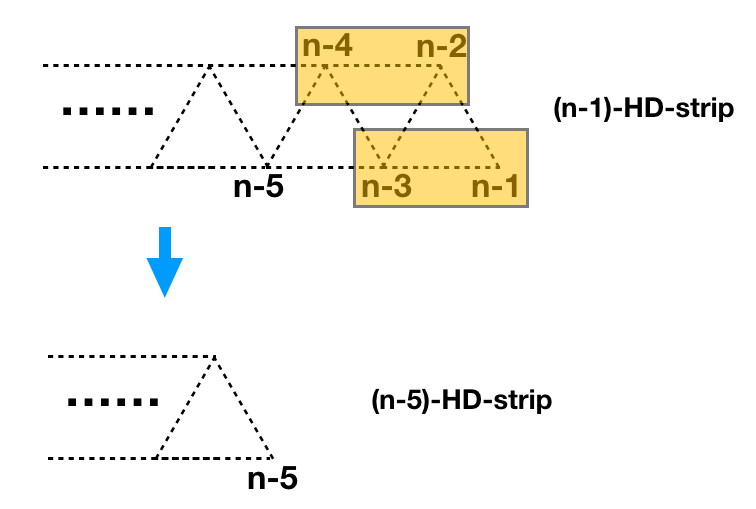}
 
 \item is a horizontal domino and a square, replace the square with a horizontal domino.

\includegraphics[width=3in]{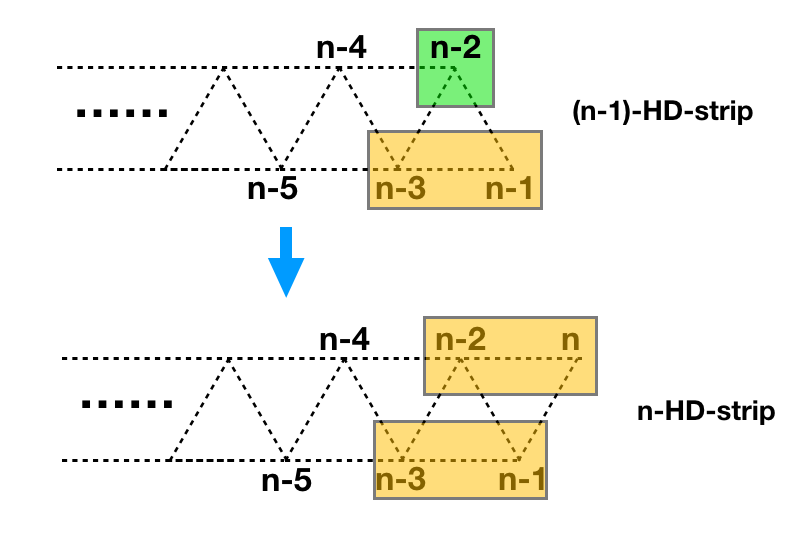}
 
 \end{enumerate}

\end{enumerate}

To verify that this is a 1-to-2 correspondence, we note that every tiling of an $n$-HD-strip or an $(n-5)$-HD-strip is indeed created exactly once by using a tiling of an  $(n-1)$-HD-strip exactly twice.

\end{proof}

The formulas in Theorem 3 and 4 were proved algebraically by Waddill [2], but here we give combinatorial proofs.

\begin{theorem}
For $n>3$, we have:
\[
T_{2n} = T_{n}^2 + T_{n-1}^2 +T_{n-2}^2+2T_{n-1}(T_{n-2} + T_{n-3}) 
\]
\end{theorem}

\begin{proof}
Question: How many tilings of a $2n$-HD-board exist?
 
Answer 1: $T_{2n}$, by definition.
 
Answer 2:
Condition on the breakability of the $n$th diagonal of a $2n$-HD-strip.
Without loss of generality, we assume that $n$ is odd.

\includegraphics[width=3in]{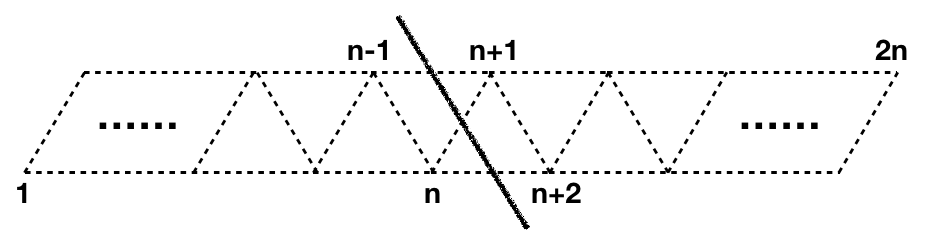}

\begin{enumerate}
  
 \item If the $n$th diagonal is breakable, then the $2n$-HD-strip can be divided into two $n$-HD-strips. Thus, there are $(T_n)^2$ ways of tilings.

 \item If it is unbreakable, there are 4 possible situations:
 \begin{enumerate}
    
 \item If an inclined domino crosses the $n$th diagonal, the 2n-HD-strip can be broken up into two $(n-1)$-HD-strips, so there are $(T_{n-1})^2$ ways of tilings.
 
 \includegraphics[width=3in]{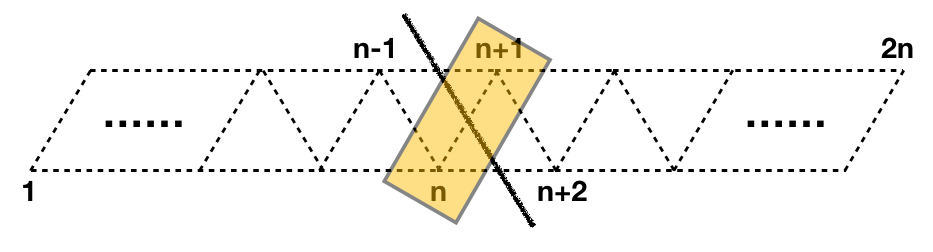}

 \item If two horizontal dominoes cross the $n$th diagonal, there are two $(n-2)$-HD-strips left, leading to $(T_{n-2})^2$ ways of tilings.
 
 \includegraphics[width=3in]{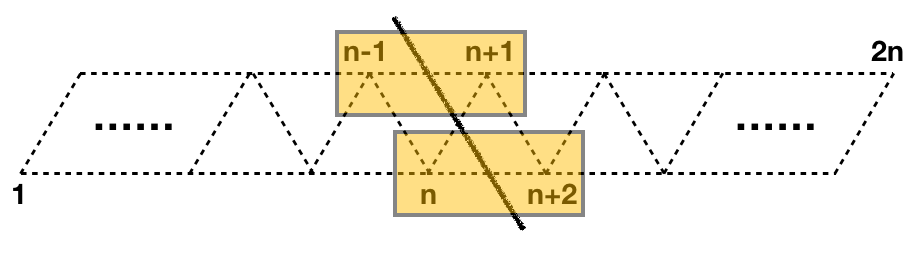}

 \item If only a horizontal domino in the upper row crosses the $n$th diagonal, further condition on the tile on cell $n$. If it is:
 \begin{enumerate}
     \item a square, there are $T_{n-2}T_{n-1}$ tilings.
     
     \includegraphics[width=3in]{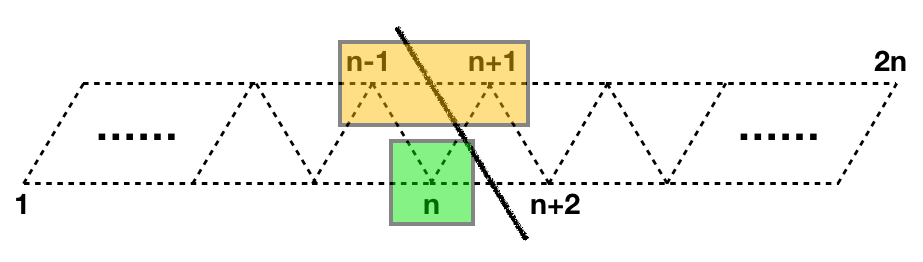}
     
     \item a horizontal domino covering cells $n$ and $n-2$, there are $T_{n-3}T_{n-1}$ tilings.
     
     \includegraphics[width=3in]{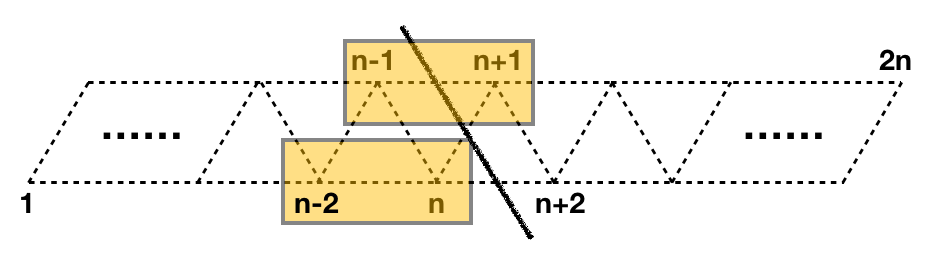}
     
 \end{enumerate}
 
 so there are $T_{n-1}(T_{n-2} + T_{n-3})$ tilings overall.

 \item If only a horizontal domino in the lower row crosses the diagonal, further condition on the tile on cell $n+1$. Similar to part c, we will get $T_{n-1}(T_{n-2} + T_{n-3})$ tilings overall.

 \end{enumerate}

\end{enumerate}
Summing the tilings in part (a) to (d) gives us an overall of $T_{n}^2 + T_{n-1}^2 +T_{n-2}^2+2T_{n-1}(T_{n-2} + T_{n-3}) $ tilings.

\end{proof}

\begin{theorem}
For $n>4$, we have:
\[
T_n -1 = T_{n-2} + 2T_{n-3} + 3(T_{n-4} + T_{n-5} + \cdots + T_1 + T_0) 
\]
\end{theorem}

\begin{proof}
Question: How many tilings of an $n$-HD-board have at least one domino?
 
Answer 1: $T_n-1$, subtracting the all-square tiling from the set of all tilings of an $n$-HD-strip.
 
Answer 2: Condition on the location of the first domino. Note that it can be either a horizontal domino, left-domino, or right-domino. Also note that the location can be an integer in the set [2,n], and all cells with cell numbers smaller than the location of the first domino are squares. When the location is:
\begin{enumerate}
  \item 2, the domino can only be a right-domino. So there are $T_{n-2}$ tilings.
  
  \includegraphics[width=3in]{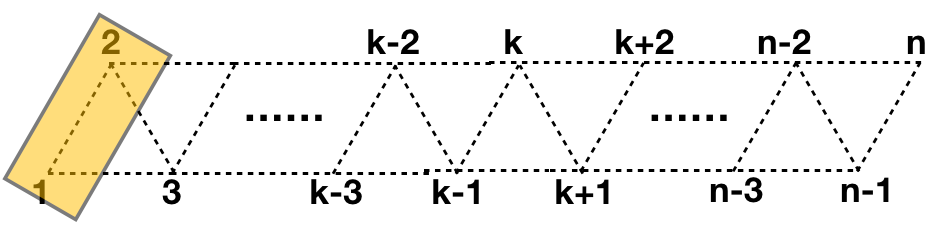}
  
  \item $k$ with $2<k<n$, the domino can be either a horizontal domino or an inclined domino (left-domino when $k$ is odd and right-domino when $k$ is even). When it is:
  \begin{enumerate}
      \item an inclined domino, there are $T_{n-k}$ tilings.
      
      \includegraphics[width=3in]{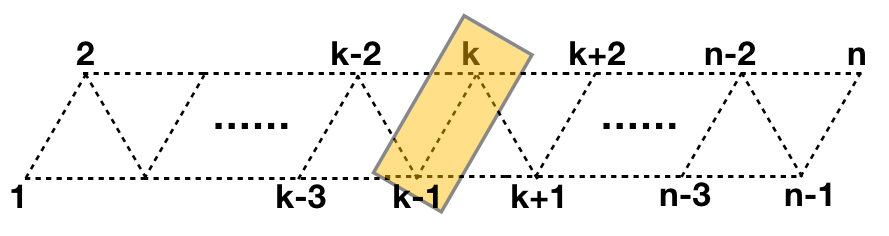}
      
      \item a horizontal domino, further condition on the tile covering cell $k-1$. If it is a square, there are $T_{n-k}$ tilings. If it is a domino, it can only be a horizontal domino covering cells $k-1$ and $k+1$, so there are $T_{n-k-1}$ tilings.
      
      \includegraphics[width=3in]{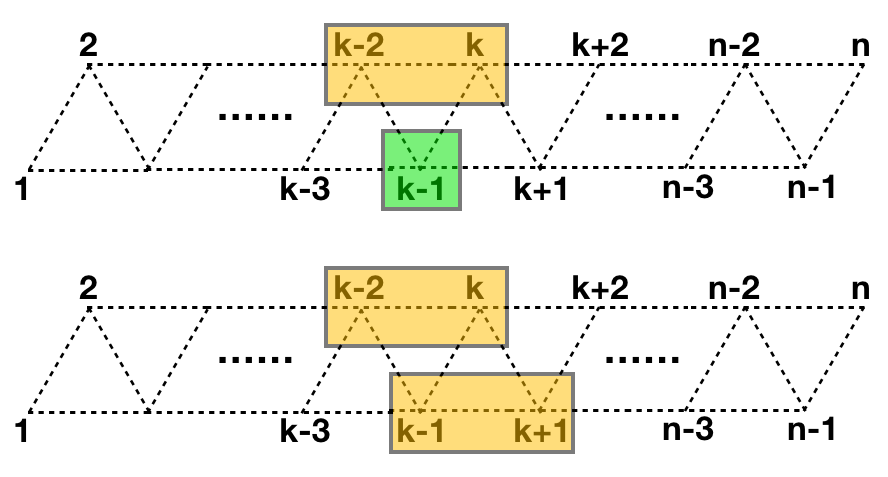}
  \end{enumerate}
  \item $n$, when the domino is:
  \begin{enumerate}
      \item a horizontal domino, the tile covering cell $n-1$ must be a square, so there are $T_0$ tilings.
      
      \includegraphics[width=3in]{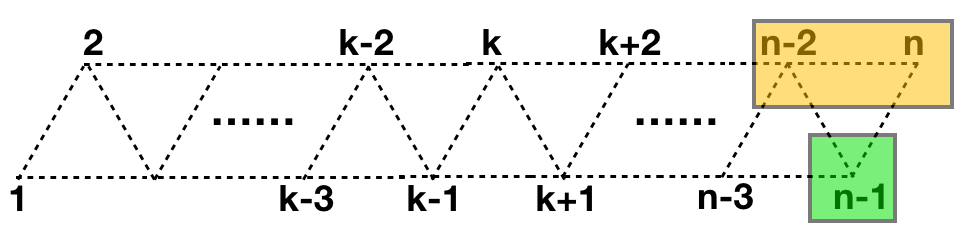}
      
      \item an inclined domino, there are also $T_0$ tilings.
      
      \includegraphics[width=3in]{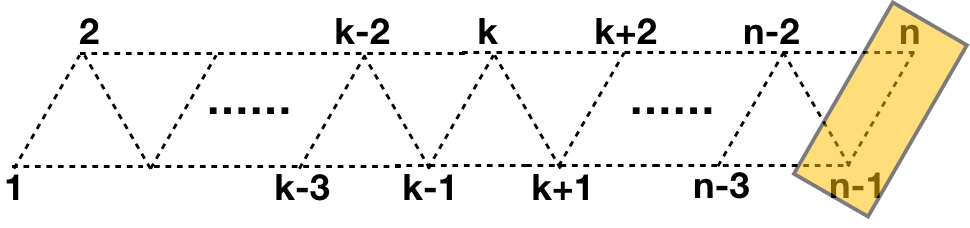}
  \end{enumerate}

\end{enumerate}

Therefore, overall there are: 

   $ T_{n-2}+ (T_{n-3} +T_{n-3} +T_{n-4}) +  (T_{n-4} +T_{n-4} +T_{n-5}) + \dots +  (T_1 +T_1 +T_0) + T_0 + T_0 = T_{n-2} + 2T_{n-3} + 3(T_{n-4} + T_{n-5} + \cdots + T_1 + T_0)$ tilings. 

\end{proof}

\section{New Proofs of New Theorems}

The following theorems appear to be completely new, as we have not found anything similar in the mathematical literature. 

We begin with a lemma.

\newtheorem{lemma}
{Lemma} 
\begin{lemma}

The number of tilings of a $2n$-HD-strip that consist of only right-dominoes and squares equals $2^n$.

\end{lemma}

\begin{proof}
we first need to draw the 2th, 4th,..., $(2n-2)$th diagonals, which decompose the $2n$-HD-strip into $n$ ``double-cells". Since there are only squares and right-dominoes, all the diagonals drawn are breakable. There are $2$ options to fill each of the $n$ double-cells: a right-domino or two squares. Therefore, there are $2^n$ tilings of the $2n$-HD-strip made up by only squares and right-dominoes.

\includegraphics[width=3in]{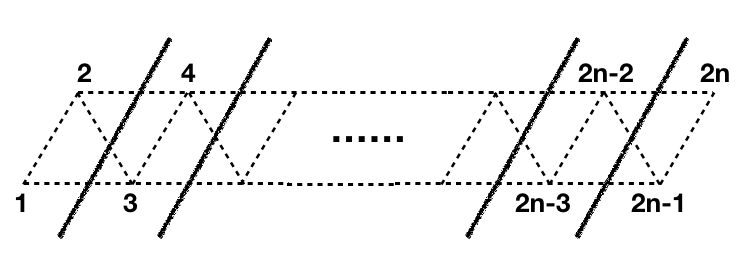}

\end{proof}

\begin{theorem}
For $n>2$, we have
\[
T_{2n} - 2^n= 2T_{n-3}+ \sum_{i=1}^{n-1} 2^{i}T_{2n-2i-2} + 5*\sum_{i=0}^{n-3} 2^{i}T_{2n-2i-5}
\]
\end{theorem}

\begin{proof}
Question: How many $2n$-HD-strips have at least one horizontal domino OR left-domino?
 
Answer 1: $T_{2n}-2^n$. By the lemma above, we subtract the set of $2n$-HD-strips with only right-dominoes and squares from the set of all $2n$-HD-strips.

Answer 2: Condition on the first horizontal domino OR left-domino. Note that the location can range from $3$ to $2n$. If the location is: \begin{enumerate}
    \item $2k$ with $1<k<n$, it must be the location of the first horizontal domino. We find that the tile on cell $2k-3$ must be a square, because it cannot be a right-domino (no space). We further condition on the tile on cell $2k-1$. If it is:
    
    \begin{enumerate}
        \item a square, the $2n$-HD-strip is decomposed into a $(2k-4)$-HD-strip on the left and a $(2n-2k)$-HD-strip on the right. There are $2^{k-2}$ ways to tile the $(2k-4)$-HD-strip by the lemma above and $T_{2n-2k}$ ways to tile the $(2n-2k)$-HD-strip. So there are $2^{k-2}T_{2n-2k}$ tilings.
        
        \includegraphics[width=3in]{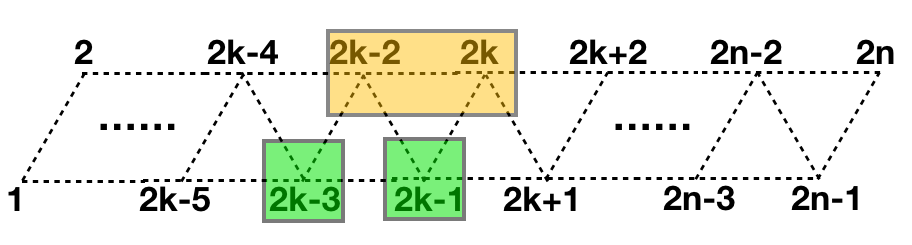}
        
        \item a horizontal domino covering cells $2k-1$ and $2k+1$, similarly, the $2n$-HD-strip is decomposed into a $(2k-4)$-HD-strip on the left and a $(2n-2k-1)$-HD-strip on the right. There are $2^{k-2}$ ways to tile the $(2k-4)$-HD-strip by the lemma above and $T_{2n-2k-1}$ ways to tile the $(2n-2k-1)$-HD-strip. So there are $2^{k-2}T_{2n-2k-1}$ tilings.
        
        \includegraphics[width=3in]{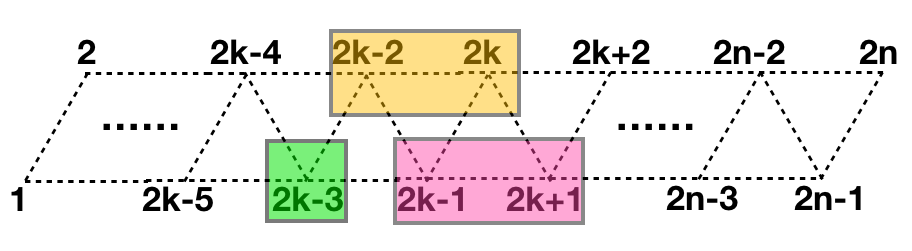}
        
    \end{enumerate}
    \item $2k-1$ with $1<k<n+1$, the tile can be either a left-domino or a horizonal domino. If it is: 
    \begin{enumerate}
        \item a horizontal domino, if the tile covering cell $2k-2$ is:
        \begin{enumerate}
            \item a square, the $2n$-HD-strip is decomposed into a $(2k-4)$-HD-strip on the left and a $(2n-2k+1)$-HD-strip on the right. By lemma 1, there are $2^{k-2}T_{2n-2k+1}$ tilings.
            
            \includegraphics[width=3in]{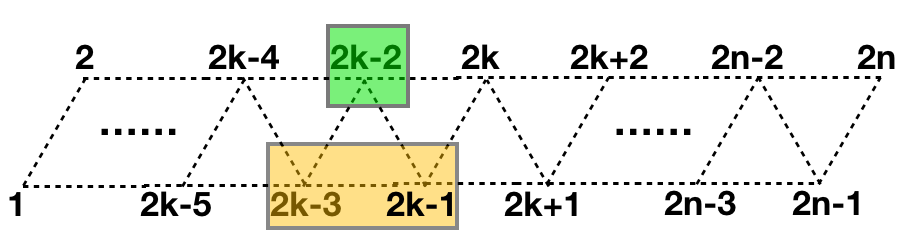}
            
            \item a horizontal domino, similarly, we would have $2^{k-2}T_{2n-2k}$ tilings.
            
            \includegraphics[width=3in]{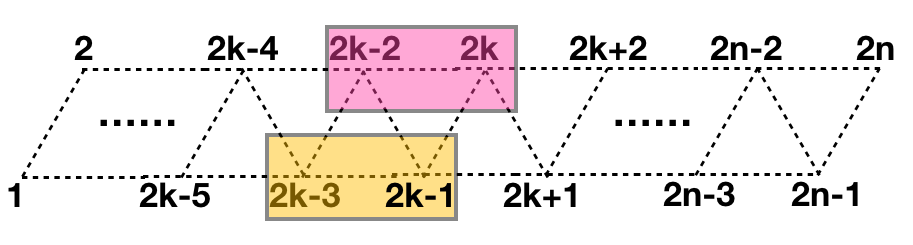}
            
        \end{enumerate}
        \item a left-domino, the tile on cell $2k-3$ must be a square, so we have a $(2k-4)$-HD-strip on the left and a $(2n-2k+1)$-HD-strip on the right. There are $2^{n-2}T_{2n-2k+1}$ tilings.
        
        \includegraphics[width=3in]{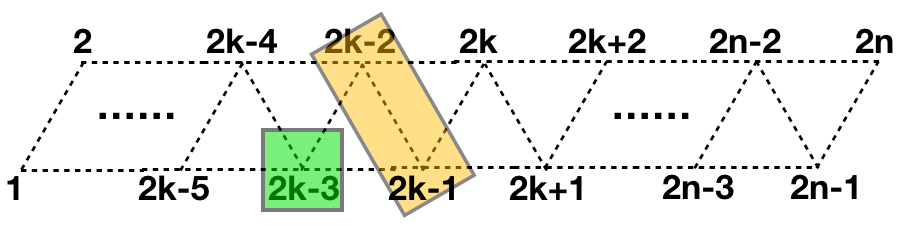}
        
    \end{enumerate}
    
    \item $2n$, the tile here must be a horizontal domino, and the tile on cell $2n-1$ must be a square. Also note that the tile on cell $2n-3$ must be a square as well, since there is no space for a right-domino. Thus, we have a $(2n-4)$-HD-strip on the left. By lemma 1, there are $2^{n-2}$ tilings, which can be written as $2^{n-2}T_0$ to align its form with other terms.
    
    \includegraphics[width=3in]{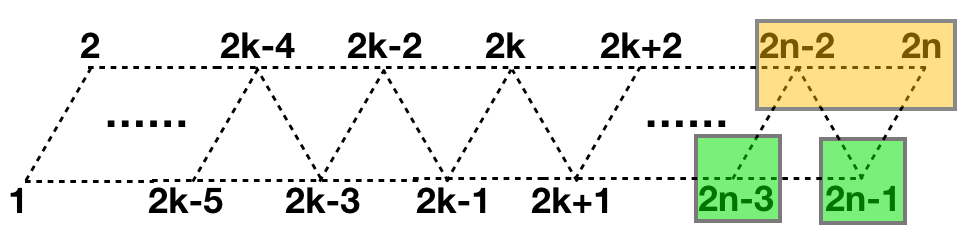}
\end{enumerate}

Adding all the parts together, we have that for each $k$ with $1<k<n$, there are $2^{k-1}T_{2n-2k} + 2^{k-1}T_{2n-2k+1} + 2^{k-2}T_{2n-2k-1}$ tilings, and that there are $2^{n-2}T_0$ tilings for location $2n$ and $2^{n-1}T_1+2^{n-2}T_0$ tilings for location $2n-1$. After uniting like terms, we get $2T_{n-3}+ \sum_{i=1}^{n-1} 2^{i}T_{2n-2i-2} + 5*\sum_{i=0}^{n-3} 2^{i}T_{2n-2i-5}$ tilings.
\end{proof}

Now, we begin to incorporate Fibonacci numbers into the formulas and show their interactions with Tetranacci numbers. One may wonder how the Fibonacci sequence is related to the HD-strips. Now we will show  that under certain conditions the tilings of HD-boards are combinatorial representations of Fibonacci numbers. We begin with two simple identities.

\newtheorem{lemma2}
{lemma2}
\begin{lemma}

If we define $H_n$ to be the number of ways to tile an n-HD-board without horizontal dominoes, we have 
\[
H_n = f_n
\]
\end{lemma}

\begin{proof}
The proof is quite simple. Without horizontal dominoes, we can ``stretch" the $n$-HD-board into an $n$-single strip. Thus, each tiling of an $n$-single strip corresponds to a tiling of an $n$-HD-board without horizontal dominoes. So $H_n$=$f_n$.
\end{proof}

\newtheorem{lemma3}
{lemma3}
\begin{lemma}
If we define $D_n$ to be the number of ways to tile a 2n-HD-board with $n$ dominoes, we have
\[
D_n = f_n
\]
\end{lemma}

\begin{proof}
Taking a closer look, we would discover that the case is in fact similar to using squares and dominoes to tile an $n$-single strip. First, there cannot be any left-dominoes, otherwise we would have odd number of cells on both sides of the left-domino, which is impossible for an all-domino tiling. Hence, we are only using right-dominoes and horizontal dominoes. The problem can be reduced to only two cases: right-stacked horizontal dominoes and right-dominoes.

On a $2n$-HD-board, right-stacked horizontal dominoes as a whole correspondes to a domino on a $n$-single strip, and a right-domino correspond to a square on a n-single strip. As a result, each of the tiling of a $2n$-HD-strip with dominoes correspond to a square-domino tiling of an $n$-single strip. So $D_n$=$f_n$.

\includegraphics[width=3.5in]{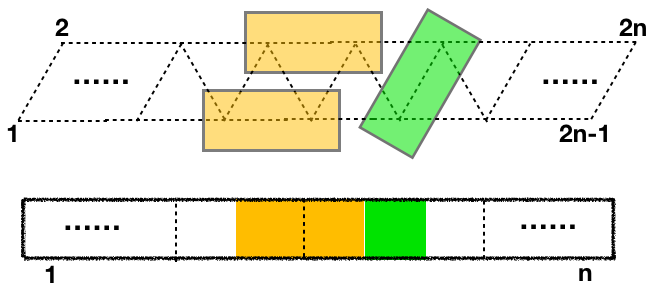}
\end{proof}

Now, we use the two lemmas to prove the next two more complex theorems.

\begin{theorem}
For $n>2$, we have:
\[
T_{2n}-f_n = T_{2n-1} f_1 + T_{2n-3} f_2 + \cdots + T_3 f_{n-1} + T_1 f_{n}
\]
or in other words,
\[
T_{2n} - f_n = \sum_{i=1}^{n} T_{2n+1-2i} f_i
\]
\end{theorem}

\begin{proof}

Question: How many tilings of an 2n-HD-board have at least one square?
 
Answer 1: $T_{2n}-f_n$. By definition, there are $T_{2n}$ tilings of an $2n$-HD-board, and by lemma 3, there are $f_n$ tilings of an all-domino $2n$-HD-board.
 
Answer 2: Condition on the location of the first square of a $2n$-HD-strip. It can take every integer value from $1$ to $2n-1$. When the location is:
\begin{itemize}
    \item $1$, there are $T_{2n-1}$ tilings, which can also be written as $T_{2n-1}f_1$ tilings in order to align its form with the remaining terms.
    
    \includegraphics[width=3in]{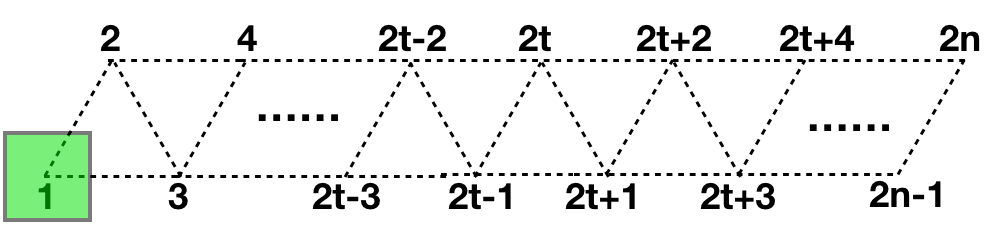}
    
    \item $k$ with $1<k<2n$, further condition on the parity of $k$. If $k$ equals:
    \begin{itemize}
        \item $2t$ with $0<t<n$, since k is the location of the first square, with a couple of trials we can find that there must be a horizontal domino that lies on cells $2t-1$ and $2t+1$, as illustrated by the graph. Hence, there is an all-domino $(2t-2)$-HD-strip on the left and a square-domino $(n-2t-1)$-HD-strip on the right. From lemma 3, we can figure out that there are $f_{t-1}T_{n-2t-1}$ , namely $f_{t-1}T_{2n-2t-1}$ tilings.
        
        \includegraphics[width=3in]{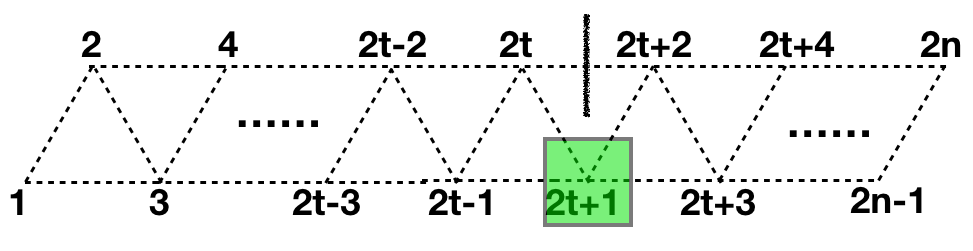}
        
        \item $2t+1$ with $0<t<n$, there cannot be a horizontal domino that covers cells $2t$ and $2t+2$, otherwise there are odd number of cells on the left, which is impossible for an all-domino tiling. Thus, the $2n$-HD-strip can be decomposed into an all-domino $2t$-HD-strip and a square-domino $(2n-2t-1)$-HD-strip. By lemma 3, there are $f_tT_{2n-2t-1}$ tilings.
        
        \includegraphics[width=3in]{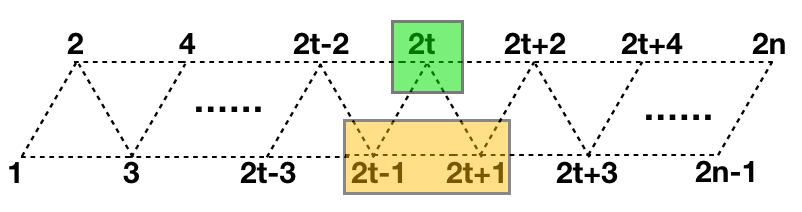}
        
        Since $f_t+f_{t-1}=f_{t+1}$, for each $t$, there are $f_{t+1}T_{n-2t-1}$ tilings. So overall there are $f_2T_{2n-3} + f_3T_{2n-5}+ \cdots + f_nT_1$ tilings when $1<k<2n$.
    \end{itemize}
\end{itemize}
Adding the two parts together, we have $T_{2n-1} f_1 + T_{2n-3} f_2 + \cdots + T_3 f_{n-1} + T_1 f_{n}$ tilings.

\end{proof}

\begin{theorem}
For $n>4$, we have:
\[
T_n - f_n = f_1T_{n-3} + f_2T_{n-4} + f_3T_{n-5} + \dots + f_{n-2}T_0
\]
or in other words, 
\[
T_n - f_n = \sum_{i=1}^{n-2}f_iT_{n-i-2}
\]
\end{theorem}

\begin{proof}
Question: How many tilings of an $n$-HD-board have at least one horizontal domino?
 
Answer 1: $T_n - f_n$. By definition, there are $T_n$ tilings of an $n$-HD-board, and by lemma 2, there are $f_n$ tilings of an $n$-HD-board without horizontal dominoes.
 
Answer 2: Condition on the location of the first horizontal domino. Note that the location can be any integer value from $3$ to $n$. For the first horizontal domino with location k, we further condition on the tile on cell $k-1$. If it is:

\begin{enumerate}
    \item a square, the $n$-HD-strip is decomposed into a $(k-3)$-HD-strip on the left and an $(n-k)$-HD-strip on the right. By lemma 2, there are $f_{k-3}$ tilings on the left and $T_{n-k}$ tilings on the right. Therefore, we have $f_{k-3}T_{n-k}$ tilings for each $k$.
    
    \includegraphics[width=3in]{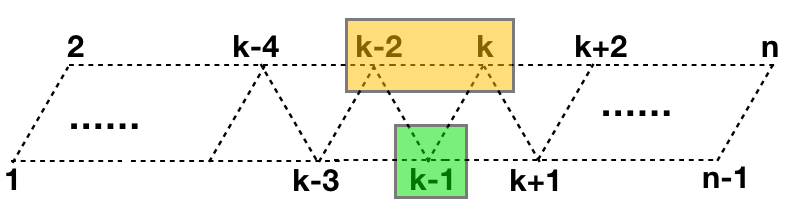}
    
    \item a horizontal domino (when $k \ne n$) covering cells $k-1$ and $k+1$ ,the $n$-HD-strip is decomposed into a $(k-3)$-HD-strip on the left and an $(n-k-1)$-HD-strip on the right. By lemma 2, there are $f_{k-3}$ tilings on the left and $T_{n-k-1}$ tilings on the right. Thus, we have $f_{k-3}T_{n-k-1}$ tilings for each $k$.
    
    \includegraphics[width=3in]{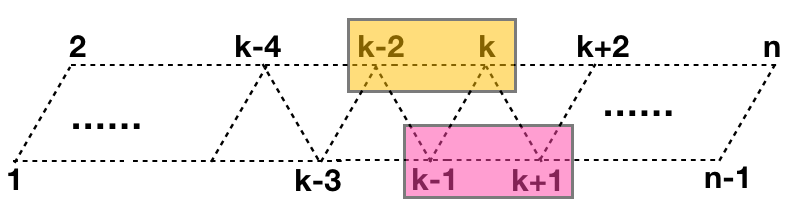}

\end{enumerate}

Overall, there are $(f_0T_{n-3} + f_1T_{n-4} + f_2T_{n-5} + \dotsc + f_{n-3}T_0) + (f_0T_{n-4} + f_1T_{n-5} +f_2T_{n-6} + \dots + f_{n-4}T_0)$ tilings. Since $f_n = f_{n-1} + f_{n-2}$ by definition, we get $f_1T_{n-3} + f_2T_{n-4} + f_3T_{n-5} + \dots + f_{n-2}T_0$ after merging the terms two by two and rewriting $f_0T_{n-3} $ as $f_1T_{n-3}$.

\end{proof}

\begin{theorem}
For $n>2$, we have:
\[
T_{2n} - (f_n)^2= \sum_{i=1}^{n} f_{i-1}^2 T_{2n-2i} + \sum_{i=2}^{n} f_{i-2} f_{i-1} T_{2n-2i+1}
\]
\end{theorem}

\begin{proof}
Question: How many tilings of a $2n$-HD-board with at least one inclined domino are there?
 
Answer 1: $T_{2n} - (f_n)^2$. Without inclined dominoes, the $2n$-HD-board can be horizontally separated into two n-single strips, which have $(f_n)^2$ tilings. Subtract them from the set of all $2n$-HD-board tilings will result in $T_{2n} - (f_n)^2$ tilings.
 
Answer 2: Condition on the location of the first inclined domino. The first inclined domino can be either a left-domino or a right-domino. If it is a:
\begin{enumerate}
    \item right-domino with location $2k$ ($0<k<n+1$), there are $(f_{k-1})^2$ tilings on the left of the right-domino and $T_{2n-2k}$ tilings on its right. Thus, there are $(f_{k-1})^2T_{2n-2k}$ tilings overall for each $k$.
    
    \includegraphics[width=3in]{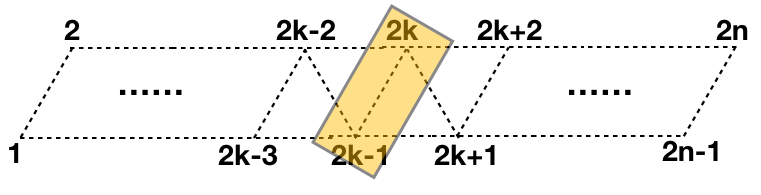}
    
    \item left-domino with location $2k-1$ ($1<k<n+1$), there are $f_{k-2}f_{k-1}$ tilings on the left of the left-domino and $T_{2n-2k+1}$ tilings on its right. So there are $(f_{k-2})(f_{k-1})T_{2n-2k+1}$ tilings overall for each $k$.
    
    \includegraphics[width=3in]{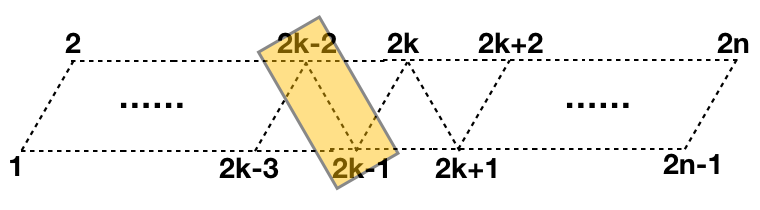}
    
\end{enumerate}

Similarly, we have: 
\[
T_{2n+1} - f_nf_{n+1} = \sum_{i=1}^{n} f_{i-1}^2 T_{2n-2i+1} + \sum_{i=1}^{n} f_{i-1} f_{i} T_{2n-2i+2}
\]
and the proof deals with an $(2n+1)$-HD-strip, using the same approach as the one above.

\end{proof}
\section{Acknowledgements}
I would like to offer my special thanks to Professor Dresden for inspiring my ideas about formulas, helping me find relevant articles for reference, and checking my writing and format. I also want to thank the Pioneer Academics Program for the opportunity to complete this research project.

\section{References}

[1] A. T. Benjamin and J. J. Quinn, \emph{Proofs That Really Count-The Art of Combinatorial Proof}, Mathematical
Association of America, Washington, DC, 2003, 8–9.
\ 

\noindent
[2] Marcellus E. Waddill, \emph{Tetranacci Sequence and Generalizations}, The Fibonacci Quarterly, 30.1 (1992), 9-19.
\ 

\noindent
[3] F. T. Howard and C. Cooper, \emph{Some Identities for r-Fibonacci Numbers}, The Fibonacci Quarterly 49.3
(2011), 231–243.
\ 

\noindent
[4] A. T. Benjamin and C. R. Heberle, \emph{Counting On r-Fibonacci Numbers}, The Fibonacci Quarterly 52.2 (2014), 121-128.
\ 

\noindent
[5] N. J. A. Sloane (Ed.), \emph{The On-Line Encyclopedia of Integer Sequences (2008)},
\href{https://oeis.org/}{https://oeis.org/}

\end{document}